%% file: main.tex
\documentclass[11pt]{amsart}

\usepackage[T1]{fontenc}
\usepackage{lmodern}
\usepackage{microtype}
\usepackage{mathtools}
\usepackage{amssymb}
\usepackage{booktabs}
\usepackage{enumitem}
\usepackage{needspace}
\usepackage{xcolor}
\usepackage{tikz}
\usepackage{float}
\usetikzlibrary{positioning,fit,backgrounds}
\usepackage[hidelinks]{hyperref}
\usepackage[capitalise,noabbrev]{cleveref}

\newtheorem{theorem}{Theorem}[section]
\newtheorem{proposition}[theorem]{Proposition}
\newtheorem{corollary}[theorem]{Corollary}
\newtheorem{lemma}[theorem]{Lemma}
\theoremstyle{definition}
\newtheorem{definition}[theorem]{Definition}

\theoremstyle{remark}
\newtheorem{remark}[theorem]{Remark}

\newcommand{\tauD}{\tau_{D}}

\newcommand{\HY}{\mathsf{B}_{\mathrm{HY}}}
\newcommand{\margin}{\mu}

\title[Counterexamples to the Henning--Yeo Conjecture]
{Counterexamples to the Henning--Yeo Conjecture: Unbounded Fixed-Degree Gaps and Sharp First-Order Asymptotics}

\author{Yufeng Wang}
\thanks{Independent researcher. Email: \href{mailto:yufeng.wang.research@gmail.com}{yufeng.wang.research@gmail.com}.}
\date{August 19, 2026}


\begin{document}
\raggedbottom

\begin{abstract}
Henning and Yeo conjectured an upper bound on the identifying vertex cover
number of a graph in terms of its order, size, and maximum degree.  A
two-parameter family $H_{t,r}$ of connected diameter-two graphs disproves the
bound for every maximum degree at least four; after denominators are cleared,
its margin is exactly $-(t-1)(r-1)$.  The complement relation
$\tauD=n-\rho$ exposes the mechanism: diameter-two fibres admit at most one
packing vertex, while degree deficit accumulates under tree gluing with
controlled port loads.  Writing $A_\Delta$ for the supremal additive gap at
maximum degree exactly $\Delta$, an exact transfer formula gives
$A_\Delta=+\infty$ for every $\Delta\ge4$, using Petersen fibres in degrees
four and five and the original $H_{t,r}$ blocks in higher degrees.  If
$c_\Delta$ denotes the corresponding supremal gap per vertex, rooted
rook-graph fibres match a universal square-graph packing bound to first order.
Consequently, $c_\Delta\sim1/\Delta$, equivalently
$\lim_{\Delta\to\infty}\Delta c_\Delta=1$.
\end{abstract}

\maketitle

\section{Introduction}

An identifying vertex cover is a vertex set that meets every edge and has a
distinct nonempty intersection with every edge; its minimum cardinality is
denoted by $\tauD(G)$.  A packing is a set of vertices at pairwise distance at
least three, and $\rho(G)$ denotes its maximum cardinality.  Henning and Yeo
credit Moncel's thesis with the first observation that an identifying vertex
cover is precisely the complement of a packing~\cite[p.~2]{HenningYeo2012};
see also~\cite{Moncel2005}.  Thus
\[
   \tauD(G)+\rho(G)=|V(G)|,
\]
for every graph $G$.

For a graph of order $n$, size $m$, and maximum degree $\Delta$, Henning and
Yeo conjectured in 2012~\cite[Conjecture~3]{HenningYeo2012} that
\begin{equation}\label{eq:HY}
  \tauD(G)
  \le
  \frac{\Delta(\Delta-1)n+2m}{\Delta^2+1}.
\end{equation}
Write $\HY(G)$ for the right-hand side of \eqref{eq:HY}.  Henning and Yeo
proved the inequality for regular graphs and for graphs with maximum degree at
most three.  They also exhibited connected graphs whose identifying vertex
cover numbers approach $\HY(G)$~\cite{HenningYeo2012}.  These results made the
general bound plausible and left maximum degree four as the first unresolved
case.

Several later works studied related edge-separation parameters.  These include
distinguishing
transversals in hypergraphs, whose 2-uniform specialization is identifying
vertex cover~\cite{HenningYeo2014}; Tracking Set Systems, whose 2-uniform case
allows at most one empty trace and is therefore a relaxation
\cite{BanikEtAl2020}; and incidence generators, which distinguish edge pairs
without requiring every edge to be met and are related to
2-packings~\cite{BozovicEtAl2022}.  To the best of our knowledge, none of these
works addresses the numerical bound in \eqref{eq:HY}, and no prior published
work or publicly available preprint gives a counterexample to Conjecture~3.

The deficit form of \eqref{eq:HY} gives the design principle for the paper:
keep packing capacity small while allowing degree deficit to accumulate.
Diameter-two fibres achieve the first goal, and controlled port loads achieve
the second.  A single violating block then raises two severity questions: can
negative margin accumulate while the maximum degree stays fixed, and how large
can the violation be per vertex?  The classical gluing operations and packing
formulas used to answer these questions are attributed in
\cref{sec:loaded-trees}; our contributions are their exact Henning--Yeo margins
and the resulting extremal consequences.

Our contributions are as follows.
\begin{enumerate}[label=(\roman*),leftmargin=2.2em]
  \item We construct connected diameter-two graphs $H_{t,r}$ whose margin,
  after clearing denominators, is $-(t-1)(r-1)$.  Consequently,
  \eqref{eq:HY} fails for every maximum degree $\Delta\ge 4$.
  \item An exact loaded-tree transfer gives connected families whose additive
  gaps are unbounded for every fixed maximum degree $\Delta\ge4$.  Thus no
  degree-dependent constant correction repairs the bound.
  \item For the supremal normalized gap $c_\Delta$ over connected graphs of
  maximum degree exactly $\Delta$, rooted rook-graph fibres give a matching
  first-order lower bound, and we prove
  \[
    \lim_{\Delta\to\infty}\Delta c_\Delta=1.
  \]
  \item A direct exhaustive check based on the definition, covering every
  graph of order at most seven, shows that the order-eight graph $H_{2,2}$ is
  order-minimal.
\end{enumerate}
The construction and extremal estimates are proved symbolically and do not
depend on computation.  Computation is used only for the final minimal-order
statement.

\paragraph{Proof roadmap.}
The complement lemma and deficit identity first reveal what a counterexample
must do, and the diameter-two blocks answer the existence question.  The
loaded-tree transfer then answers whether negative margin can accumulate at
fixed maximum degree.  Finally, a square-graph upper bound and a rooted rook
construction meet at the normalized scale.  The exhaustive computation is
logically independent and is used only to locate the smallest witness.

\section{Preliminaries and a deficit reformulation}

All graphs in this paper are finite, simple, and undirected; for the null graph
we take the maximum degree to be zero.  For a graph $G$, write $n=|V(G)|$,
$m=|E(G)|$, and $\Delta=\Delta(G)$.  A set $T\subseteq V(G)$ is an
\emph{identifying vertex cover} if $e\cap T$ is nonempty for every
$e\in E(G)$ and the sets $e\cap T$ are pairwise distinct.  The minimum
cardinality of such a set is $\tauD(G)$.

A \emph{packing} is a set $S\subseteq V(G)$ whose distinct vertices have
pairwise distance at least three; its maximum cardinality is $\rho(G)$.  We
use the convention that vertices in different components have distance
$+\infty$.  We include the standard complement relation for completeness.

\begin{lemma}[Moncel~\cite{Moncel2005}; see also Henning--Yeo~\cite{HenningYeo2012}]\label{lem:complement}
A set $T\subseteq V(G)$ is an identifying vertex cover if and only if $V(G)\setminus T$ is a packing.  Consequently,
\[
   \tauD(G)+\rho(G)=n.
\]
\end{lemma}

\begin{proof}
Let $S=V(G)\setminus T$.  If $T$ is an identifying vertex cover, then no two vertices of $S$ are adjacent.  Moreover, two vertices $u,v\in S$ cannot have a common neighbour $w$, because the edges $uw$ and $vw$ would both have intersection $\{w\}$ with $T$.  Thus distinct vertices of $S$ have distance at least three.

Conversely, suppose that $S$ is a packing and put $T=V(G)\setminus S$.  Since $S$ is independent, $T$ meets every edge.  If two distinct edges had the same intersection with $T$, that intersection would have to be a common singleton $\{w\}$; their other endpoints would then be two vertices of $S$ at distance two through $w$, a contradiction.  Hence $T$ identifies all edges.  Taking complementary minimum and maximum sets proves the identity.
\end{proof}

Define the Henning--Yeo bound and its integer margin by
\begin{align*}
  \HY(G)&=\frac{\Delta(\Delta-1)n+2m}{\Delta^2+1},\\
  \margin(G)&=\Delta(\Delta-1)n+2m-(\Delta^2+1)\tauD(G).
\end{align*}
Thus $G$ is a counterexample precisely when $\margin(G)<0$.  We call
\[
  \tauD(G)-\HY(G)=-\frac{\margin(G)}{\Delta^2+1}
\]
the \emph{additive gap}; it is positive precisely for counterexamples.

\begin{proposition}[Deficit reformulation]\label{prop:deficit}
For every graph $G$,
\begin{equation}\label{eq:margin-packing}
  \margin(G)
  =(\Delta^2+1)\rho(G)-\bigl((\Delta+1)n-2m\bigr).
\end{equation}
In particular, the conjectured inequality~\eqref{eq:HY} is equivalent to
\begin{equation}\label{eq:deficit}
  \sum_{v\in V(G)}\bigl(\Delta+1-d(v)\bigr)
  \le (\Delta^2+1)\rho(G).
\end{equation}
\end{proposition}

\begin{proof}
Substitute $\tauD(G)=n-\rho(G)$ from \cref{lem:complement} into the definition of $\margin(G)$ and simplify.  The equality
$(\Delta+1)n-2m=\sum_v(\Delta+1-d(v))$ is the handshaking lemma.
\end{proof}

Thus \eqref{eq:deficit} is a budget constraint: each packing vertex supplies
$\Delta^2+1$ units of capacity, whereas a vertex $v$ demands
$\Delta+1-d(v)$ units.  Our constructions keep the number of suppliers small
while allowing the demand to accumulate.

\section{A two-parameter family of diameter-two counterexamples}

The deficit inequality \eqref{eq:deficit} suggests a two-part design: force
$\rho=1$ by making the graph diameter two, while retaining enough low-degree
vertices for the total degree deficit to exceed $\Delta^2+1$.  The family
below realizes these requirements with two independently tunable classes of
$a$--$b$ routes; the clique on $D$ closes the remaining distance-two pairs.
Fix integers $t,r\ge 2$.  Define $H_{t,r}$ on the vertex set
\[
  \{a,b\}\cup X\cup C\cup D,
\]
where $X=\{x_1,\ldots,x_t\}$, $C=\{c_1,\ldots,c_r\}$, and $D=\{d_1,\ldots,d_r\}$.  Its edge set consists of
\begin{align*}
  &ax_i,\;bx_i &&(1\le i\le t),\\
  &ac_j,\;c_jd_j,\;d_jb &&(1\le j\le r),\\
  &d_id_j &&(1\le i<j\le r).
\end{align*}
Thus $a$ and $b$ have $t$ common neighbours, there are $r$ internally disjoint $a$--$b$ paths of length three, and the vertices in $D$ induce a clique.

\begin{theorem}[Diameter-two blocks]\label{thm:block}
For all integers $t,r\ge 2$, the graph $H_{t,r}$ is connected and has
\begin{align*}
 n(H_{t,r})&=t+2r+2,&
 m(H_{t,r})&=2t+3r+\binom r2,\\
 \Delta(H_{t,r})&=t+r,&
 \rho(H_{t,r})&=1,&
 \tauD(H_{t,r})&=t+2r+1.
\end{align*}
Moreover,
\begin{equation}\label{eq:block-margin}
   \margin(H_{t,r})=-(t-1)(r-1)<0.
\end{equation}
Hence every $H_{t,r}$ is a connected counterexample to \eqref{eq:HY}.
\end{theorem}

\begin{proof}
The order and size follow directly from the construction.  The vertex degrees are
\[
 d(a)=d(b)=t+r,\qquad d(x_i)=d(c_j)=2,\qquad d(d_j)=r+1.
\]
Since $t\ge2$, the maximum degree is $\Delta=t+r$.

We next verify the diameter systematically.  It is enough to consider nonadjacent pairs.  The pair $a,b$ has the length-two path $a x_1 b$.  Two vertices of $X$, or two vertices of $C$, have the common neighbour $a$; vertices of $D$ are pairwise adjacent.  For $x_i\in X$, the paths $x_i a c_j$ and $x_i b d_j$ handle pairs with $C$ and $D$, respectively.  The paths $a c_j d_j$ and $b d_j c_j$ handle pairs in $\{a\}\times D$ and $\{b\}\times C$.  Finally, $c_i d_i$ is an edge, while for $i\ne j$ the path $c_i d_i d_j$ has length two because $D$ is a clique.  Every nonadjacent pair appears in this list, so the diameter is at most two.  Since $a$ and $b$ are not adjacent, the diameter is exactly two.

A diameter-two graph has no packing with two vertices, so $\rho(H_{t,r})=1$.  By \cref{lem:complement}, $\tauD(H_{t,r})=n-1=t+2r+1$.  Now apply \cref{prop:deficit} with $\Delta=t+r$:
\begin{align*}
 \margin(H_{t,r})
 &=((t+r)^2+1)-\bigl((t+r+1)(t+2r+2)-2m(H_{t,r})\bigr)\\
 &=-(t-1)(r-1).
\end{align*}
This is negative for $t,r\ge2$.
\end{proof}

The factorization in \eqref{eq:block-margin} is more informative than its
sign: the two tunable parts of the construction interact multiplicatively,
so the violation is neither tied to one degree nor caused by a numerical
accident.

\begin{corollary}\label{cor:all-degrees}
For every integer $\Delta\ge4$, there is a connected counterexample to \eqref{eq:HY} with maximum degree exactly $\Delta$.
\end{corollary}

\begin{proof}
Choose any $t,r\ge2$ with $t+r=\Delta$, for example $t=2$ and $r=\Delta-2$, and apply \cref{thm:block}.
\end{proof}

\begin{figure}[!htbp]
\centering
\begin{tikzpicture}[
  scale=0.9,
  every node/.style={circle,draw,inner sep=1.8pt,font=\small},
  every edge/.style={draw}
]
  \node (a) at (-3,0) {$a$};
  \node (b) at (3,0) {$b$};
  \node (x1) at (0,0.8) {$x_1$};
  \node (x2) at (0,-0.8) {$x_2$};
  \node (c1) at (-1.8,2) {$c_1$};
  \node (c2) at (-1.8,-2) {$c_2$};
  \node (d1) at (1.8,2) {$d_1$};
  \node (d2) at (1.8,-2) {$d_2$};
  \begin{pgfonlayer}{background}
    \node[shape=rectangle,draw=gray!65,dashed,rounded corners=4pt,fill=gray!7,fit=(c1)(c2),inner sep=8pt] {};
    \node[shape=rectangle,draw=gray!65,dashed,rounded corners=4pt,fill=gray!7,fit=(x1)(x2),inner sep=8pt] {};
    \node[shape=rectangle,draw=gray!65,dashed,rounded corners=4pt,fill=gray!7,fit=(d1)(d2),inner sep=8pt] {};
  \end{pgfonlayer}
  \node[shape=rectangle,draw=none,fill=white,inner sep=1pt,font=\scriptsize\itshape] at (-1.8,2.55) {$C$};
  \node[shape=rectangle,draw=none,fill=white,inner sep=1pt,font=\scriptsize\itshape] at (0,1.35) {$X$};
  \node[shape=rectangle,draw=none,fill=white,inner sep=1pt,font=\scriptsize\itshape] at (1.8,2.55) {$D$};
  \draw (a)--(x1)--(b) (a)--(x2)--(b);
  \draw (a)--(c1)--(d1)--(b);
  \draw (a)--(c2)--(d2)--(b);
  \draw (d1)--(d2);
\end{tikzpicture}
\caption{The smallest block $H_{2,2}$.  The common neighbours in $X$ keep
$a$ and $b$ at distance two, the clique on $D$ closes the remaining two-step
routes, and the degree-two vertices in $X\cup C$ supply degree deficit.  In
general there are $t$ common neighbours $x_i$ and $r$ paths $a c_j d_j b$.}
\label{fig:h22}
\end{figure}
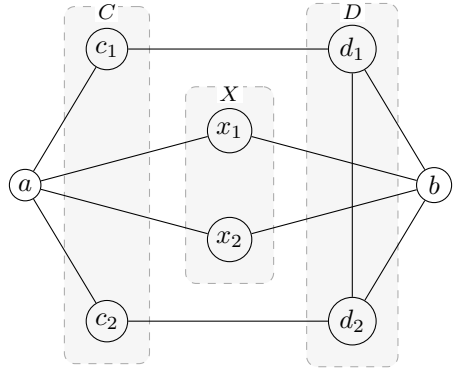

The rational amount by which $H_{t,r}$ exceeds the conjectured bound is
\begin{equation}\label{eq:block-gap}
  \tauD(H_{t,r})-\HY(H_{t,r})
  =\frac{(t-1)(r-1)}{(t+r)^2+1}.
\end{equation}
For the diagonal family $t=r=s$, this quantity tends to $1/4$.  A single block may therefore resemble a rounding obstruction.  The next construction shows that repeated total deficit, rather than rounding, is the structural source of failure.

\input{sections/loaded_trees}

\input{sections/exact_limit}

\section{The smallest witness: an exact computational check}
\label{sec:smallest-witness}

We now turn to an independent finite question: how small can any
counterexample be?  Computation first located the pattern, but the family proof
above explains it; exhaustive search is used here only to certify the left
endpoint at order eight.  The graph $H_{2,2}$ has order eight, size eleven,
and degree sequence
$(4,4,3,3,2,2,2,2)$; an isomorphic labelled representative, written in the
standard \texttt{graph6} format, has encoding
\[
  \texttt{GXFGXC}.
\]
An explicit relabeling from the construction to vertices $0,\ldots,7$ of this
representative is
\begin{align*}
a&\mapsto2,& b&\mapsto5,& x_1&\mapsto0,& x_2&\mapsto1,\\
c_1&\mapsto3,& c_2&\mapsto7,& d_1&\mapsto4,& d_2&\mapsto6.
\end{align*}
It has $\rho=1$ and $\tauD=7$, whereas \eqref{eq:HY} would require
\[
  17\cdot7\le4\cdot3\cdot8+2\cdot11,
\]
that is, $119\le118$.

An isomorphic copy of $H_{2,2}$ also appears as a smallest graph of fault
cost~$3$ in Goedgebeur et al.~\cite[Proposition~4 and
Figure~14]{GoedgebeurEtAl2026} and as House of Graphs record
53055~\cite{HouseOfGraphs53055}.  In the House of Graphs labelling, an explicit
isomorphism is
\[
 (a,b,x_1,x_2,c_1,c_2,d_1,d_2)
 \longmapsto(6,7,0,1,2,3,4,5).
\]
That work concerns minimum-leaf spanning structures and does not address
identifying vertex covers, packing numbers, or \eqref{eq:HY}.  The new content
here is the identification and proof of this graph as an order-minimal
counterexample to \eqref{eq:HY}, together with the family and fixed-degree
extensions above.

\begin{proposition}[Computer-assisted]\label{prop:minimal}
No graph of order at most seven is a counterexample to \eqref{eq:HY}.  Consequently, $H_{2,2}$ has minimum possible order among all counterexamples.
\end{proposition}

\begin{proof}[Verification protocol]
We checked the proposition by exhaustive subset enumeration on two separately
sourced catalogues.  The first consists of the 1,253 nonisomorphic graphs of
order at most seven in the Read--Wilson atlas~\cite{ReadWilson1998}, accessed
through NetworkX 3.5~\cite{HagbergEtAl2008,NetworkXAtlas35}.  For the second,
we used Brendan McKay's six order-specific \texttt{graph6} files for orders two
through seven, comprising 1,251 records, and added the unique graphs of orders
zero and one locally~\cite{McKayGraphData}.  Record counts and input hashes are
fixed in the reproducibility archive.

For every graph, one checker enumerated all vertex subsets and tested the
identifying-cover definition directly.  A separate enumeration computed the
packing number and checked $\tauD+\rho=n$.  A third implementation, using only
the Python standard library, decoded the \texttt{graph6} records and repeated
both calculations.  The checkers returned no negative margin on
the 1,253 graphs and all returned margin $-1$ on $H_{2,2}$.  They also verify
the explicit relabeling from $H_{2,2}$ to the \texttt{GXFGXC} representative
given above.  All parameter and bound comparisons use integer arithmetic.
\end{proof}

The accompanying reproducibility archive (Supplementary File S1, version
0.3) is supplied as
\begin{center}
\small\ttfamily
identifying-vertex-cover-\\[-1pt]
supplement-v0.3-reproducible.zip
\end{center}
It contains the code, input digests, portable outputs, dependency
specification, regression tests, and a SHA-256 manifest.  Its top-level
SHA-256 digest is
\begin{center}
\scriptsize\ttfamily
488a630ec7492857f89283d7e34ccfc1cf56cc1bb9f8c48db8b5f57114f3bbc9
\end{center}
The computational proposition is not needed for
\cref{thm:block,thm:all-fixed-degree-unbounded,thm:exact-c-delta-limit}; it
only proves order minimality.  We make no claim that $H_{2,2}$ is the unique
counterexample of order eight.

\section{Conclusion}

The Henning--Yeo inequality fails at three distinct scales.  The
diameter-two blocks $H_{t,r}$ give connected counterexamples for every maximum
degree at least four.  Paths of $H_{t,r}$ blocks, together with loaded trees
of Petersen fibres at the two critical degrees, make the additive violation
unbounded at every fixed maximum degree $\Delta\ge4$.  Thus neither rounding
nor any finite degree-dependent additive correction repairs the bound.
Across all three scales, the same mechanism is at work: diameter-two fibres
cap packing capacity at one vertex, whereas degree deficit accumulates almost
linearly under load-controlled gluing.

At the normalized scale, rooted rook-graph fibres provide
$c_\Delta\ge\Delta^{-1}-6\Delta^{-4/3}$, while the square-graph packing
argument supplies a matching first-order upper bound.  Consequently,
\[
  \lim_{\Delta\to\infty}\Delta c_\Delta=1.
\]
This determines the first-order severity of the failure, not merely its
asymptotic order.

The exact values of $c_\Delta$ at fixed degree and the global second-order
term remain open.  It would also be useful to characterize extremal or
asymptotically extremal connected graphs, and to determine which structural
features beyond the loaded-fibre mechanism can approach the first-order
constant.

\section*{AI-assistance disclosure}
Generative-AI tools were used extensively for exploratory search, proposing
candidate constructions and proof arguments, code development, assistance
with literature searches, drafting, and language editing.  No AI system is
an author.  The human author independently reconstructed the arguments,
checked all mathematical claims, computations, citations, and wording, and
accepts sole responsibility for the manuscript.

\bibliographystyle{amsplain}
\bibliography{refs}

\end{document}

%% file: sections/loaded_trees.tex
\section{Loaded trees and unbounded fixed-degree violation}
\label{sec:loaded-trees}

A single block proves failure but does not yet show that its deficit can
accumulate while the maximum degree remains fixed.  We therefore glue
diameter-two fibres along a tree: the quotient tree makes the bridge
contribution exact, port loads record the increase in maximum degree, and
unused vertices provide one packing choice per fibre.  Allowing bridge
incidences at one fibre to use either distinct vertices or the same vertex
will distinguish the two critical degrees.

\subsection{An exact transfer theorem}

\begin{definition}[Loaded tree of fibres]\label{def:loaded-tree}
Let $T$ be a tree on $[k]$, where $k\ge2$, and let
$R_1,\ldots,R_k$ be vertex-disjoint copies of a graph $R$.  For every
incidence $(i,e)$ with $i\in e\in E(T)$, choose a vertex
$p_{i,e}\in V(R_i)$, called its \emph{port}.  The \emph{load} of
$v\in V(R_i)$ is
\[
  \lambda_i(v)=|\{e\in E(T):i\in e\text{ and }p_{i,e}=v\}|.
\]
The maximum port load is $L=\max_{i,v}\lambda_i(v)$.  A vertex
$s_i\in V(R_i)$ with $\lambda_i(s_i)=0$ is an \emph{unused selector}.
For every edge $e=ij$ of $T$, add the edge $p_{i,e}p_{j,e}$; the resulting
graph is denoted by $G(T,R,p)$.
\end{definition}

The next theorem records the exact effect of this operation on every
parameter in the Henning--Yeo bound.  In particular, $L$ is the actual
maximum load, not merely an upper bound.

\begin{theorem}[Loaded-tree transfer]\label{thm:loaded-tree}
Let $R$ be a connected $r$-regular graph of order $q$ and diameter at most
two.  Let $G=G(T,R,p)$ be a loaded tree of $k\ge2$ copies of $R$.  Suppose
that the maximum port load is exactly $L$ and that every copy contains an
unused selector.  Then $G$ is connected and
\begin{align*}
 |V(G)|&=qk,&
 |E(G)|&=\frac{rqk}{2}+k-1,&
 \Delta(G)&=r+L,\\
 \rho(G)&=k,&
 \tauD(G)&=(q-1)k.
\end{align*}
Moreover,
\begin{equation}\label{eq:loaded-tree-margin}
 \margin(G)
 =k\bigl((r+L)^2+3-(L+1)q\bigr)-2,
\end{equation}
and hence
\begin{equation}\label{eq:loaded-tree-gap}
 \tauD(G)-\HY(G)
 =\frac{k\bigl((L+1)q-(r+L)^2-3\bigr)+2}
 {(r+L)^2+1}.
\end{equation}
\end{theorem}

\begin{proof}
The added edges follow the tree $T$, so $G$ is simple and connected.  The
order formula is immediate, and the $k-1$ tree edges give the stated size.
Every $v\in V(R_i)$ has degree $r+\lambda_i(v)$.  Since the maximum load is
attained, the maximum degree is exactly $r+L$.

Because each $R_i$ has diameter at most two, a packing in $G$ contains at
most one vertex from each copy.  Conversely, choose an unused selector
$s_i$ in every copy.  A path from $s_i$ to $s_j$, where $i\ne j$, must use
an edge inside $R_i$ before its first bridge, at least one bridge, and an
edge inside $R_j$ after its last bridge.  Thus its length is at least three,
so $\{s_1,\ldots,s_k\}$ is a packing.  Therefore $\rho(G)=k$, and
\cref{lem:complement} gives $\tauD(G)=qk-k$.

Put $D=r+L$.  Substitution in \cref{eq:margin-packing} yields
\begin{align*}
 \margin(G)
 &=(D^2+1)k-
   \left((D+1)qk-\left(rqk+2(k-1)\right)\right)\\
 &=k\bigl(D^2+3-(L+1)q\bigr)-2.
\end{align*}
This is \eqref{eq:loaded-tree-margin}; negating it and dividing by $D^2+1$
gives \eqref{eq:loaded-tree-gap}.
\end{proof}

\subsection{Fixed degrees}

The transfer formula is also a design criterion.  For any scalable family
satisfying \cref{thm:loaded-tree}, the additive gap has positive linear slope
precisely when
\[
   (L+1)q>(r+L)^2+3.
\]
The Petersen graph $P$ has $(q,r)=(10,3)$ and diameter two.  It clears this
threshold for both $L=1$ and $L=2$, with excess $1$ and $2$, respectively.
Thus distinct ports give maximum degree four, whereas reusing one root gives
maximum degree five.

\Needspace{12\baselineskip}
\begin{corollary}[The critical degrees]\label{cor:petersen-critical}
The following connected Petersen-fibre families have unbounded additive
violation.
\begin{enumerate}[label=\textup{(\alph*)},leftmargin=2.2em]
 \item For every $k\ge2$, join $k$ Petersen copies over a path, assigning
 the two incidences of each internal copy to distinct ports.  The resulting
 graph $P_{4,k}$ satisfies
 \[
  (n,m,\Delta,\rho,\tauD)=(10k,16k-1,4,k,9k)
 \]
 and
 \[
  \margin(P_{4,k})=-k-2,
  \qquad
  \tauD(P_{4,k})-\HY(P_{4,k})=\frac{k+2}{17}.
 \]
 \item For every $k\ge3$, join $k$ Petersen copies over a path, using one
 fixed root for every incidence in a copy.  The resulting graph $P_{5,k}$
 satisfies
 \[
  (n,m,\Delta,\rho,\tauD)=(10k,16k-1,5,k,9k)
 \]
 and
 \[
  \margin(P_{5,k})=-2k-2,
  \qquad
  \tauD(P_{5,k})-\HY(P_{5,k})=\frac{k+1}{13}.
 \]
\end{enumerate}
\end{corollary}

\begin{proof}
In part~\textup{(a)}, the actual maximum load is one; an internal copy uses
two distinct ports, and a third vertex can be chosen as its selector.  Apply
\cref{thm:loaded-tree} with $(r,q,L)=(3,10,1)$.

In part~\textup{(b)}, an internal copy exists because $k\ge3$, and its root
has load two; any other vertex is an unused selector.  Apply
\cref{thm:loaded-tree} with $(r,q,L)=(3,10,2)$.
\end{proof}

\begin{remark}[Provenance]
Bridge gluing already occurs in the examples of Henning and
Yeo~\cite{HenningYeo2012}.  When all incidences at a fibre use one prescribed
root, this is the classical rooted product~\cite{GodsilMcKay1978}, whose
packing number is covered by a general product formula~\cite{MojdehEtAl2020}.
The path-union framework is due to Shee and Ho~\cite{SheeHo1996}, and Vaidya
and Kanani later used the repeated-root Petersen path
\cite[Theorem~2.5, p.~68]{VaidyaKanani2011}.  Our use of these classical
operations is to determine their exact Henning--Yeo margin and its consequences.
\end{remark}

For higher maximum degree, no new fibre is needed.  Each copy of
$H_{3,\Delta-3}$ contributes cleared margin $-2(\Delta-4)$, while every new
bridge contributes $2$; the net linear coefficient is therefore
$-2(\Delta-5)$.  This explains why the original blocks take over at
$\Delta=6$.  For $\Delta\ge6$ and $k\ge1$, let $C_{\Delta,k}$ be obtained
from $k$ copies of $H_{3,\Delta-3}$ by joining consecutive vertices
$c_1^{(i)}c_1^{(i+1)}$ along a path, as in the block-chain construction.

\begin{proposition}\label{prop:large-degree-chains}
For every $\Delta\ge6$ and $k\ge1$, the graph $C_{\Delta,k}$ is connected,
has maximum degree $\Delta$, and satisfies
\[
 \rho(C_{\Delta,k})=k,
 \qquad
 \margin(C_{\Delta,k})=-2k(\Delta-5)-2.
\]
Consequently,
\begin{equation}\label{eq:large-degree-gap}
 \tauD(C_{\Delta,k})-\HY(C_{\Delta,k})
 =\frac{2k(\Delta-5)+2}{\Delta^2+1}.
\end{equation}
\end{proposition}

\begin{proof}
Every bridge endpoint has degree at most four, whereas the vertices $a$ and
$b$ in each block retain degree $3+(\Delta-3)=\Delta$.  Each block has
diameter two by \cref{thm:block}, so a packing uses at most one vertex per
block.  The vertices corresponding to $a$ in the $k$ copies are pairwise at
distance $|i-j|+2$, and hence form a packing of size $k$.  Finally, the sum
of the $k$ block margins from \cref{eq:block-margin}, together with the
$k-1$ bridges, gives
\[
 -2k(\Delta-4)+2(k-1)=-2k(\Delta-5)-2.
\]
Equation~\eqref{eq:large-degree-gap} follows from the definition of the
additive gap.
\end{proof}

To state the fixed-degree conclusion, define the extended-real quantity
\[
 A_\Delta=\sup\{\tauD(G)-\HY(G):
   G\text{ is connected and }\Delta(G)=\Delta\}.
\]

\begin{theorem}[Unbounded failure at every fixed degree]
\label{thm:all-fixed-degree-unbounded}
For every integer $\Delta\ge4$,
\[
  A_\Delta=+\infty.
\]
\end{theorem}

\begin{proof}
The two parts of \cref{cor:petersen-critical} give the assertion for
$\Delta=4$ and $\Delta=5$, respectively.  For every $\Delta\ge6$,
\eqref{eq:large-degree-gap} tends to infinity with $k$.
\end{proof}

\begin{corollary}\label{cor:no-additive-correction}
For no $\Delta\ge4$ is there a finite constant $C_\Delta$ such that
\[
  \tauD(G)\le\HY(G)+C_\Delta
\]
for every connected graph $G$ of maximum degree exactly $\Delta$.  Hence no
universal fixed additive correction works, and replacing $\HY(G)$ by
$\lceil\HY(G)\rceil$ does not repair the conjecture.
\end{corollary}

\begin{proof}
The first two assertions are immediate from
\cref{thm:all-fixed-degree-unbounded}.  For the last, choose a member of any
one fixed-degree family whose additive gap is greater than one; then
$\tauD(G)>\HY(G)+1\ge\lceil\HY(G)\rceil$.
\end{proof}

%% file: sections/exact_limit.tex
\section{The extremal normalized gap}\label{sec:extremal-gap}

For an integer $\Delta\ge 4$, define
\begin{equation}\label{eq:c-delta-definition}
 c_\Delta=
 \sup\left\{
   \frac{\tauD(G)-\HY(G)}{|V(G)|}:
   G\text{ is connected and }\Delta(G)=\Delta
 \right\}.
\end{equation}
The maximum-degree condition in this definition is an equality.  We first
record a universal upper bound, and then construct graphs whose normalized
gaps have the same first-order asymptotic behaviour.

Henning and Yeo's general estimate~\cite[Theorem~7]{HenningYeo2012} already
implies
\[
  \limsup_{\Delta\to\infty}\Delta c_\Delta\le1.
\]
We prove the following stronger finite-degree bound because it gives a
quantitative refinement.

\subsection{A universal upper bound}

\begin{proposition}[Square-graph upper bound]\label{prop:c-delta-upper}
For every integer $\Delta\ge4$,
\begin{equation}\label{eq:c-delta-upper}
 c_\Delta\le
 \frac{\bigl(\sqrt{\Delta+\Delta^{-1}}-1\bigr)^2}
      {\Delta^2+1}.
\end{equation}
Consequently,
\[
 \limsup_{\Delta\to\infty}\Delta c_\Delta\le1.
\]
\end{proposition}

\begin{proof}
Let $G$ have order $n$, average degree $\bar d$, and maximum degree
$\Delta$.  In the square graph $G^2$, two distinct vertices are adjacent when
their distance in $G$ is at most two.  Thus packings in $G$ are precisely
independent sets in $G^2$.
 Choose a uniformly random ordering of $V(G^2)$ and retain each vertex that
 precedes all of its neighbours.  The retained set is independent, and a
 vertex $v$ is retained with probability $1/(d_{G^2}(v)+1)$.  Hence the
 Caro--Wei argument~\cite{Caro1979,Wei1981}, followed by the
 arithmetic--harmonic mean inequality, gives
\[
 \rho(G)
 \ge \sum_{v\in V(G)}\frac{1}{d_{G^2}(v)+1}
 \ge \frac{n^2}{n+\sum_v d_{G^2}(v)}.
\]
 For fixed $v$, map a neighbour $u$ to the incidence $(u,v)$; for every
 vertex $w$ at distance two, choose a path $vuw$ and map $w$ to $(u,w)$.
 Different target vertices give different incidences, all counted by
 $\sum_{u\in N_G(v)}d_G(u)$, so
\[
 d_{G^2}(v)\le\sum_{u\in N_G(v)}d_G(u).
\]
Summing first over $v$ and then over their neighbours yields
\[
 \sum_v d_{G^2}(v)
 \le\sum_v d_G(v)^2
 \le\Delta\sum_vd_G(v)=\Delta n\bar d.
\]
It follows that $\rho(G)/n\ge(1+\Delta\bar d)^{-1}$.  Substitution in
the margin identity~\eqref{eq:margin-packing} gives
\begin{equation}\label{eq:gap-average-degree-upper}
 \frac{\tauD(G)-\HY(G)}{n}
 \le
 \frac{\Delta+1-\bar d-(\Delta^2+1)/(1+\Delta\bar d)}
      {\Delta^2+1}.
\end{equation}
The numerator on the right is strictly concave in $\bar d$.  Its unique
maximum on $[0,\Delta]$ occurs at
\[
\bar d=\sqrt{\Delta+\Delta^{-1}}-\Delta^{-1},
\]
which lies in $[0,\Delta]$ for $\Delta\ge4$.  The maximum value is
$\bigl(\sqrt{\Delta+\Delta^{-1}}-1\bigr)^2$.  Taking the supremum in
\eqref{eq:c-delta-definition} proves \eqref{eq:c-delta-upper}.  Multiplying
its right-hand side by $\Delta$ and taking the limit proves the final
assertion.
\end{proof}

\subsection{Designing a matching lower bound}

We use an elementary family of diameter-two fibres.  For $s\ge2$, the
\emph{rook graph} $R_s=K_s\mathbin{\square}K_s$ has vertex set
$[s]\times[s]$, with two distinct pairs adjacent when they agree in one
coordinate.

\begin{lemma}\label{lem:rook-parameters}
The graph $R_s$ has order $q=s^2$, is $r$-regular with
$r=2(s-1)$, and has diameter two.
\end{lemma}

\begin{proof}
There are $s^2$ ordered pairs.  From $(i,j)$, the $s-1$ other vertices in
row $i$ and the $s-1$ other vertices in column $j$ are its distinct
neighbours, so the degree is $2(s-1)$.  Two vertices that agree in a
coordinate are adjacent.  If $(i,j)$ and $(i',j')$ differ in both
coordinates, then
\[
 (i,j),(i,j'),(i',j')
\]
is a path of length two.  Such nonadjacent pairs exist for $s\ge2$, and
therefore the diameter is exactly two.
\end{proof}

The loaded-tree formula also explains the choice below.  For an $r$-regular
diameter-two fibre of order $q$, with root load $L=\Delta-r$, its limiting
normalized gap is
\[
 \frac{\Delta-r+1}{\Delta^2+1}
 -\frac{\Delta^2+3}{(\Delta^2+1)q}.
\]
Thus a first-order sharp construction should keep $r$ small compared with
$\Delta$ while making $q$ large.  The rook graphs provide the explicit tradeoff
$r=2(s-1)$ and $q=s^2$.  After multiplication by $\Delta$, the two losses from
the target value $1$ have orders $s/\Delta$ and $\Delta/s^2$; balancing them
gives $s\asymp\Delta^{2/3}$.

Fix now an integer $\Delta\ge8$, and put
\begin{equation}\label{eq:rook-choice}
 s=\left\lfloor\Delta^{2/3}\right\rfloor,
 \qquad q=s^2,
 \qquad r=2(s-1),
 \qquad L=\Delta-r=\Delta-2s+2.
\end{equation}
Here $s\ge4$.  Since $s\le\Delta^{2/3}$ and
$\Delta\ge2\Delta^{2/3}$ for $\Delta\ge8$, we also have $L\ge2$.

For every integer $k\ge L+1$, choose a tree $T_{L,k}$ of order $k$ and
maximum degree exactly $L$.  Such trees exist: start with $K_{1,L}$ and,
when $k>L+1$, attach the additional vertices successively beyond one leaf.
Take one copy of $R_s$ for
each vertex of $T_{L,k}$, choose a root in every copy, and, for every edge of
$T_{L,k}$, join the two corresponding roots.  Denote the resulting graph by
$R_{\Delta,k}$.  Equivalently, this is the rooted product of $T_{L,k}$ with
the rooted rook graph.

\begin{lemma}[Rooted rook family]\label{lem:rook-rooted-family}
For every $\Delta\ge8$ and every $k\ge L+1$, the graph $R_{\Delta,k}$ is
connected, has maximum degree exactly $\Delta$, and satisfies
\begin{align*}
 |V(R_{\Delta,k})|&=qk,&
 |E(R_{\Delta,k})|&=\frac{rqk}{2}+k-1,\\
 \rho(R_{\Delta,k})&=k,&
 \tauD(R_{\Delta,k})&=(q-1)k.
\end{align*}
\Needspace{8\baselineskip}
Moreover,
\begin{equation}\label{eq:rook-normalized-gap-finite-k}
 \frac{\tauD(R_{\Delta,k})-\HY(R_{\Delta,k})}
      {|V(R_{\Delta,k})|}
 =\phi_\Delta+
   \frac{2}{(\Delta^2+1)s^2k},
\end{equation}
where
\begin{equation}\label{eq:phi-delta}
 \phi_\Delta=
 \frac{\Delta-2s+3}{\Delta^2+1}
 -\frac{\Delta^2+3}{(\Delta^2+1)s^2}.
\end{equation}
\end{lemma}

\begin{proof}
This is the repeated-root specialization of \cref{thm:loaded-tree}; we spell
out the parameters because exact maximum degree and finite $k$ are part of
the definition of $c_\Delta$.

The order and size formulas follow from the $k$ disjoint fibres and the
$k-1$ tree edges.  A root over a vertex $x$ of $T_{L,k}$ has degree
$r+d_{T_{L,k}}(x)$, while every non-root vertex has degree $r$.  Because
$\Delta(T_{L,k})=L$, the actual maximum degree of the product is
$r+L=\Delta$.

Every fibre has diameter two, so a packing contains at most one vertex from
each fibre.  Conversely, choose one non-root vertex in every fibre.  A path
between choices in distinct fibres must travel from the first choice to its
root, traverse at least one tree edge, and travel from the second root to the
second choice.  Its length is therefore at least three.  Hence
 $\rho(R_{\Delta,k})=k$, and the complement identity gives
 $\tauD(R_{\Delta,k})=(q-1)k$.  This special case also follows from the
published packing formula for rooted products
\cite[Theorem~5.1]{MojdehEtAl2020}.

Using the actual maximum degree $\Delta=r+L$ and the displayed parameters,
the cleared margin is
\begin{align*}
 \margin(R_{\Delta,k})
 &=\Delta(\Delta-1)qk+rqk+2(k-1)
   -(\Delta^2+1)(q-1)k\\
 &=k\bigl(\Delta^2+3-(L+1)q\bigr)-2.
\end{align*}
Negating, dividing by $(\Delta^2+1)qk$, and substituting
$q=s^2$ and $L+1=\Delta-2s+3$ proves
\eqref{eq:rook-normalized-gap-finite-k} and \eqref{eq:phi-delta}.
\end{proof}

Because $R_{\Delta,k}$ is a finite connected graph with maximum degree
\emph{exactly} $\Delta$ for every $k\ge L+1$, each member is admissible in
\eqref{eq:c-delta-definition}.  There are admissible values of $k$ tending
to infinity, so \eqref{eq:rook-normalized-gap-finite-k} and the definition of
a supremum give
\begin{equation}\label{eq:c-at-least-phi}
 c_\Delta\ge
 \lim_{k\to\infty}
 \frac{\tauD(R_{\Delta,k})-\HY(R_{\Delta,k})}
      {|V(R_{\Delta,k})|}
 =\phi_\Delta.
\end{equation}
The limit in \eqref{eq:c-at-least-phi} is not asserted to be attained by a
finite graph.

\begin{theorem}[Exact first-order asymptotics]\label{thm:exact-c-delta-limit}
For every integer $\Delta\ge8$,
\begin{equation}\label{eq:quantitative-c-lower}
 c_\Delta\ge \frac1\Delta-6\Delta^{-4/3}.
\end{equation}
Furthermore,
\begin{equation}\label{eq:exact-c-limit}
 \lim_{\Delta\to\infty}\Delta c_\Delta=1.
\end{equation}
\end{theorem}

\begin{proof}
For $s=\lfloor\Delta^{2/3}\rfloor$, direct simplification of
\eqref{eq:phi-delta} gives the exact scaled-error identity
\begin{equation}\label{eq:exact-scaled-error}
 1-\Delta\phi_\Delta
 =\frac{\Delta(2s-3)+1}{\Delta^2+1}
  +\frac{\Delta(\Delta^2+3)}{(\Delta^2+1)s^2}.
\end{equation}
Both terms are positive.  Since $\Delta\ge8$ and $s\ge4$,
\[
 \frac{\Delta(2s-3)+1}{\Delta^2+1}
 \le \frac{2s}{\Delta}.
\]
 Since $s=\lfloor\Delta^{2/3}\rfloor$, we have
 $\Delta^{2/3}<s+1$, and hence $\Delta^2<(s+1)^3$.  Moreover $s\ge4$ gives
 $(s+1)^3<2s^3$.  Also
$(\Delta^2+3)/(\Delta^2+1)<2$.  Therefore
\[
 \frac{\Delta(\Delta^2+3)}{(\Delta^2+1)s^2}
 <\frac{2\Delta}{s^2}
 <\frac{4s}{\Delta}.
\]
Together with $s\le\Delta^{2/3}$, equation
\eqref{eq:exact-scaled-error} yields
\[
 0<1-\Delta\phi_\Delta
 <\frac{6s}{\Delta}
 \le6\Delta^{-1/3}.
\]
Combining this estimate with \eqref{eq:c-at-least-phi} proves
\eqref{eq:quantitative-c-lower} and
$\liminf_{\Delta\to\infty}\Delta c_\Delta\ge1$.  The reverse inequality
$\limsup_{\Delta\to\infty}\Delta c_\Delta\le1$ is
\cref{prop:c-delta-upper}, and hence \eqref{eq:exact-c-limit} follows.
\end{proof}

\begin{remark}\label{rem:second-order-open}
The theorem determines the first-order term $c_\Delta\sim1/\Delta$ only.
The rook construction gives a scaled error of order $O(\Delta^{-1/3})$, but
the global second-order behaviour of $c_\Delta$ remains open; in particular,
no optimality of that exponent is claimed for arbitrary connected graphs.
\end{remark}

%% file: refs.bib
@article{HenningYeo2012,
  author  = {Michael A. Henning and Anders Yeo},
  title   = {Identifying Vertex Covers in Graphs},
  journal = {The Electronic Journal of Combinatorics},
  volume  = {19},
  number  = {4},
  pages   = {P32},
  year    = {2012},
  doi     = {10.37236/2114}
}

@phdthesis{Moncel2005,
  author = {Julien Moncel},
  title  = {Codes identifiants dans les graphes},
  school = {Universit{\'e} Joseph Fourier, Grenoble I},
  year   = {2005},
  url    = {https://theses.hal.science/tel-00010293}
}

@article{HenningYeo2014,
  author  = {Michael A. Henning and Anders Yeo},
  title   = {Distinguishing-Transversal in Hypergraphs and Identifying Open Codes in Cubic Graphs},
  journal = {Graphs and Combinatorics},
  volume  = {30},
  number  = {4},
  pages   = {909--932},
  year    = {2014},
  doi     = {10.1007/s00373-013-1311-2}
}

@article{BanikEtAl2020,
  author  = {Aritra Banik and Pratibha Choudhary and Venkatesh Raman and Saket Saurabh},
  title   = {Fixed-Parameter Tractable Algorithms for Tracking Shortest Paths},
  journal = {Theoretical Computer Science},
  volume  = {846},
  pages   = {1--13},
  year    = {2020},
  doi     = {10.1016/j.tcs.2020.09.006}
}

@article{BozovicEtAl2022,
  author  = {Dragana Bo{\v z}ovi{\'c} and Aleksander Kelenc and Iztok Peterin and Ismael G. Yero},
  title   = {Incidence Dimension and 2-Packing Number in Graphs},
  journal = {RAIRO. Operations Research},
  volume  = {56},
  number  = {1},
  pages   = {199--211},
  year    = {2022},
  doi     = {10.1051/ro/2022001}
}

@book{ReadWilson1998,
  author    = {Ronald C. Read and Robin J. Wilson},
  title     = {An Atlas of Graphs},
  publisher = {Oxford University Press},
  address   = {Oxford},
  year      = {1998},
  isbn      = {9780198532897},
  doi       = {10.1093/oso/9780198532897.001.0001}
}

@inproceedings{HagbergEtAl2008,
  author    = {Aric A. Hagberg and Daniel A. Schult and Pieter J. Swart},
  title     = {Exploring Network Structure, Dynamics, and Function Using {NetworkX}},
  booktitle = {Proceedings of the 7th Python in Science Conference},
  editor    = {Ga{\"e}l Varoquaux and Travis Vaught and Jarrod Millman},
  address   = {Pasadena, California},
  pages     = {11--15},
  month     = aug,
  year      = {2008},
  url       = {https://conference.scipy.org/proceedings/SciPy2008/paper_2/}
}

@misc{NetworkXAtlas35,
  author = {{NetworkX developers}},
  title  = {{graph\_atlas\_g} in {NetworkX} 3.5},
  year   = {2025},
  url    = {https://networkx.org/documentation/networkx-3.5/reference/generated/networkx.generators.atlas.graph_atlas_g.html},
  note   = {Accessed 18 August 2026}
}

@misc{McKayGraphData,
  author = {Brendan D. McKay},
  title  = {Combinatorial Data: Graphs},
  url    = {https://users.cecs.anu.edu.au/~bdm/data/graphs.html},
  note   = {Accessed 18 August 2026}
}

@article{GoedgebeurEtAl2026,
  author  = {Jan Goedgebeur and Jarne Renders and G{\'a}bor Wiener and Carol T. Zamfirescu},
  title   = {Network Fault Costs Based on Minimum Leaf Spanning Trees},
  journal = {Applied Mathematics and Computation},
  volume  = {525},
  pages   = {130057},
  year    = {2026},
  doi     = {10.1016/j.amc.2026.130057},
  eprint  = {2502.10213},
  archivePrefix = {arXiv},
  primaryClass  = {math.CO}
}

@misc{HouseOfGraphs53055,
  author = {Jarne Renders},
  title  = {House of Graphs, Graph 53055},
  year   = {2025},
  url    = {https://houseofgraphs.org/graphs/53055},
  note   = {Comment: ``A smallest graph with fault cost 3''; uploaded 10 January 2025; accessed 18 August 2026}
}

@techreport{Caro1979,
  author      = {Yair Caro},
  title       = {New Results on the Independence Number},
  institution = {Tel Aviv University},
  year        = {1979}
}

@techreport{Wei1981,
  author      = {V. K. Wei},
  title       = {A Lower Bound on the Stability Number of a Simple Graph},
  institution = {Bell Laboratories},
  number      = {TM 81-11217-9},
  address     = {Murray Hill, NJ},
  year        = {1981}
}

@article{GodsilMcKay1978,
  author  = {Godsil, C. D. and McKay, B. D.},
  title   = {A New Graph Product and Its Spectrum},
  journal = {Bulletin of the Australian Mathematical Society},
  volume  = {18},
  number  = {1},
  pages   = {21--28},
  month   = feb,
  year    = {1978},
  doi     = {10.1017/S0004972700007760},
  url     = {https://doi.org/10.1017/S0004972700007760}
}

@article{MojdehEtAl2020,
  author  = {Mojdeh, Doost Ali and Peterin, Iztok and Samadi, Babak and Yero, Ismael G.},
  title   = {{(Open) Packing Number of Some Graph Products}},
  journal = {Discrete Mathematics \& Theoretical Computer Science},
  volume  = {22},
  number  = {4},
  pages   = {Paper No. 1},
  month   = aug,
  year    = {2020},
  doi     = {10.23638/DMTCS-22-4-1},
  url     = {https://dmtcs.episciences.org/6730},
  eprint  = {1901.06813},
  archiveprefix = {arXiv},
  primaryclass  = {math.CO}
}

@article{SheeHo1996,
  author  = {Shee, Sze-Chin and Ho, Yong-Song},
  title   = {The Cordiality of the Path-Union of {$n$} Copies of a Graph},
  journal = {Discrete Mathematics},
  volume  = {151},
  number  = {1--3},
  pages   = {221--229},
  month   = may,
  year    = {1996},
  doi     = {10.1016/0012-365X(94)00099-5},
  url     = {https://doi.org/10.1016/0012-365X(94)00099-5}
}

@article{VaidyaKanani2011,
  author  = {Vaidya, S. K. and Kanani, K. K.},
  title   = {Some New Product Cordial Graphs},
  journal = {Mathematics Today},
  volume  = {27},
  pages   = {64--70},
  month   = jun,
  year    = {2011},
  issn    = {0976-3228},
  note    = {Proceedings of Maths Meet--2011; online scan accessed 19 August 2026},
  url     = {https://www.researchgate.net/publication/258024450_Some_New_Product_Cordial_Graphs}
}
